\documentclass{amsart}

\pdfoutput=1
\usepackage[utf8]{inputenc}
\usepackage{lmodern}
\usepackage[T1]{fontenc}
\usepackage{mathtools}
\allowdisplaybreaks
\usepackage{amssymb} 
\usepackage{amsthm}
\usepackage[usenames,dvipsnames,svgnames,table]{xcolor}
\usepackage{xspace}
\usepackage{embrac} 
\ChangeEmph{(}[-.01em,.04em]{)}[.04em,-.05em]
\usepackage{microtype}
\usepackage{textcase}
\usepackage{enumitem}
\usepackage[acronym]{glossaries}
\usepackage{graphicx}
\usepackage{float}
\usepackage{booktabs}
\usepackage{csquotes}
\usepackage{multirow}
\usepackage{tikz}
\usepackage{pgfplots}
\pgfplotsset{compat=1.18}
\usepackage{subcaption}
\usetikzlibrary{angles,quotes,calc,positioning,shapes,matrix,arrows}
\usepackage{tikz-cd}
\usetikzlibrary{}
\usepackage{tkz-euclide}
\usepackage{etoolbox}
\usepackage{url}
\usepackage[hyperindex=true, colorlinks=true, linkcolor=Maroon, citecolor=DarkGreen, urlcolor=NavyBlue, breaklinks=true,linktocpage=true,linktoc=all]{hyperref}
\usepackage[english,capitalise,noabbrev]{cleveref}
\crefdefaultlabelformat{#2\textup{#1}#3}
\crefname{equation}{equation}{equations}
\usepackage[logical-quotes,backrefs,initials]{amsrefs}

\newcommand{\R}{\mathbb{R}}

\renewcommand{\S}{\mathbb{S}}

\newcommand{\contained}{\subset}

\newcommand{\union}{\cup}

\newcommand{\suchthat}{\, : \,}
\newcommand{\st}{\suchthat}

\newcommand{\defined}{\coloneqq}

\newcommand{\tendsto}{\rightarrow}

\newcommand{\from}{\colon}

\newcommand{\vol}{\operatorname{vol}}

\DeclarePairedDelimiterX\norm[1]\lVert\rVert{\ifblank{#1}{\:\cdot\:}{#1}}
\DeclarePairedDelimiterX\abs[1]\lvert\rvert{\ifblank{#1}{\:\cdot\:}{#1}}
\DeclarePairedDelimiterX\covolume[1]\lvert\rvert{\ifblank{#1}{\:\cdot\:}{#1}}
\DeclarePairedDelimiterX\set[1]{\{}{\}}{\ifblank{#1}{\: \:}{#1}}
\DeclarePairedDelimiterX\innerprod[2]\langle\rangle
{\ifblank{#1#2}{\,\cdot,\cdot\,}
	{#1,#2}}
\DeclarePairedDelimiterX\floor[1]\lfloor\rfloor{\ifblank{#1}{\:\cdot\:}{#1}}
\DeclarePairedDelimiterXPP\expectation[2]{\ifblank{#1}{\mathbb{E}}{\mathbb{E}_{#1}}}\lbrack\rbrack{}{\ifblank{#2}{\:\cdot\:}{#2}}

\newcommand{\dd}{\mathop{}\!\mathrm{d}}

\newcommand{\Energy}{E}
\newcommand{\minEnergy}{\mathcal{E}}
\newcommand{\Wiener}{W}

\renewcommand{\digamma}{\psi}

\theoremstyle{plain}
\newtheorem{theorem}{Theorem}[section]

\newtheorem{conjecture}[theorem]{Conjecture}
\newtheorem{corollary}[theorem]{Corollary}

\newtheorem{lemma}[theorem]{Lemma}
\newtheorem{proposition}[theorem]{Proposition}

\theoremstyle{definition}

\newtheorem{definition}[theorem]{Definition}

\theoremstyle{remark}

\begin{document}

	\title{Riesz 2-energy of the Diamond ensemble}
	
	\author{Pedro R. López-Gómez}
	\address{Pedro R. López-Gómez: Departamento de Matemáticas, CUNEF Universidad,  C/\,Almansa, 101, 28040 Madrid, Spain}
	\email{pedro.lopezgomez@cunef.edu}
	
	\date{\today{}}
	
	\thanks{The author has been supported by grant PID2020-113887GB-I00 funded by MCIN/AEI/10.13039/501100011033 and by grant PRE2021-097772 funded by MCIN/AEI/10.13039/501100011033 and by ``ESF Investing in your future''.}
	
	

	\begin{abstract}
		We study the Riesz $2$-energy of point configurations on the two-dimensional sphere arising from the Diamond ensemble, a construction of well-distributed points introduced by Beltrán and Etayo in 2020. For this family of point sets, we derive an explicit formula for the expected $2$-energy, valid for general choices of the parameters. Using this formula, we analyze a specific realization of the Diamond ensemble and obtain the asymptotic expansion of its expected $2$-energy. As a consequence, we establish a new upper bound for the minimal Riesz $2$-energy on the sphere, improving upon all previously known upper bounds. In particular, our result yields, for the first time, a negative coefficient in the quadratic term of the asymptotic expansion, bringing the upper bound significantly closer to the conjectured constant.
	\end{abstract}
	
	\maketitle
	

\section{Introduction and main results}

The problem of distributing points uniformly on the sphere is a classical question with connections to potential theory, approximation theory, harmonic analysis, and probability theory; see \cite{SaffKuijlaars1997} for a nice introduction to the topic. Among the various criteria commonly used to address this problem, one that has attracted considerable attention in recent decades is to seek configurations that minimize a certain pairwise interaction energy. In this sense, Riesz $s$-energies play a central role; see, for example, the monograph \cite{BorodachovHardinSaff2019} or the survey \cite{BrauchartGrabner2015}. Given a collection $\omega_N=\set{x_1,\dotsc,x_N}\contained\S^d$ of $N$ points on the $d$-dimensional sphere, the \emph{Riesz $s$-energy} of $\omega_N$ is defined as
\begin{equation*}
	\Energy_s(\omega_N)\defined\sum_{i\neq j}K_s(x_i,x_j),
\end{equation*}
where $K_s$ denotes the \emph{Riesz $s$-kernel}, given by
\begin{equation*}
	K_s(x,y)\defined\begin{cases}
	\norm{x-y}^{-s}	& s>0,\\
	\log\norm{x-y}^{-1}	& s=\log,
	\end{cases}\qquad x,y\in\S^d.
\end{equation*}
Here, as usual, we are considering the logarithmic energy as the case $s=\log$ of the Riesz energy. The \emph{minimal Riesz $s$-energy of $N$ points} is then defined as
\begin{equation*}
	\minEnergy_s(\S^d,N)\defined\min\set{\Energy_s(\omega_N)\st \omega_N\contained\S^d}.
\end{equation*}
Minimizers of the Riesz $s$-energy on the sphere are known to be asymptotically uniformly distributed for all $s>0$ (in fact, this property holds for all $s>-2$, but we restrict our attention here to the repulsive regime $s>0$) and $s=\log$; see \cite[Theorem 4.4.9 and Proposition 4.6.4]{BorodachovHardinSaff2019}. Nevertheless, the asymptotic behavior of the minimal Riesz energies as the number of points tends to infinity depends strongly on the value of $s$. In this regard, one can clearly distinguish two different regimes. On the one hand, there is the \emph{singular} or \emph{potential-theoretic} regime, corresponding to $0<s<d$ and $s=\log$. For these values of $s$, the Riesz kernel is integrable and the asymptotic expansion is known to take the form
\begin{equation}\label{eq:leading-term-s-energy-singular}
	\minEnergy_s(\S^d,N)=\Wiener_s(\S^d)N^2+o(N^2),
\end{equation}
where $\Wiener_s(\S^d)$ is the \emph{Wiener constant} of the sphere, defined as the infimum of the \emph{continuous $s$-energy}
\begin{equation*}
	I_s[\mu]\defined\iint_{\S^d\times\S^d}K_s(x,y)\dd\mu(x)\dd\mu(y)
\end{equation*}
over all Borel probability measures supported on $\S^d$ (see \cite[Theorem 4.2.2]{BorodachovHardinSaff2019}). It is well known that this continuous energy is uniquely minimized by the uniform measure $\sigma$; see \cite[Proposition 4.6.4]{BorodachovHardinSaff2019}.

On the other hand, when $s\geq d$ the Riesz kernel is no longer integrable, and one has $I_s[\mu]=+\infty$ for every probability measure supported on $\S^d$. In this \emph{hypersingular} regime, short-range interactions become increasingly relevant as the parameter $s$ grows, and the order of growth of the minimal $s$-energy changes from $N^2$ to $N^2\log N$ for the critical value $s=d$, and then to $N^{1+s/d}$ for $s>d$. More specifically, thanks to the works \cites{KuijlaarsSaff1998,GotzSaff2001,HardinSaff2005} we know that the asymptotic behavior of the minimal $s$-energies in this case is the following:
\begin{align}
	\minEnergy_d(\S^d,N)&=\frac{\vol(B^d)}{\vol(\S^d)} N^2\log N+o(N^2\log N)\label{eq:leading-term-d-energy},\\
	\minEnergy_s(\S^d,N)&=\frac{C_{s,d}}{\vol(\S^d)^{s/d}} N^{1+s/d}+o(N^{1+s/d}),\qquad s>d,\label{eq:leading-term-s-energy-hypersingular}
\end{align}
where $B^d$ is the unit ball in $\R^d$ and $C_{s,d}$ is a finite constant depending only on $s$ and $d$. The explicit value of the constant $C_{s,d}$ is not known except for $d=1$, $8$, and $24$. It should be noted that these results for the hypersingular regime, which may be viewed as a particular instance of the celebrated poppy-seed bagel theorem of Hardin and Saff \cite{HardinSaff2005}, actually hold in a much broader setting; see \cite[Chapters 8 and 9]{BorodachovHardinSaff2019} for a more general discussion. The explicit expressions for the next-order terms in \cref{eq:leading-term-s-energy-singular,eq:leading-term-d-energy,eq:leading-term-s-energy-hypersingular} are not known in general, except for the logarithmic case, in which it is known to be $-N\log(N)/d$ (see \cite[Theorem 6.4.6]{BorodachovHardinSaff2019}). Therefore, improving our understanding of these terms becomes a central problem. In this paper, we focus on the case $s=2$ on the two-dimensional sphere. For conjectures concerning the next-order terms for this and other values of $s$, we refer the reader to \cite[Section  6.6]{BorodachovHardinSaff2019}.

\subsection{Riesz $2$-energy on the sphere. Conjecture and state of the art}

As mentioned, the case $s=2$ on $\S^2$ is of particular interest because, on this space, it marks the transition between the singular regime ($s<2$) and the hypersingular regime ($s> 2$). From \eqref{eq:leading-term-d-energy}, we have
\begin{equation*}
	\minEnergy_2(\S^2,N)=\frac{1}{4} N^2\log N+o(N^2\log N).
\end{equation*}
Regarding the next-order terms for this energy, Brauchart, Hardin, and Saff proposed in \cite{BrauchartHardinSaff2012} the following conjecture for the asymptotic expansion of the minimal $2$-energy on the sphere; see \cite[Section 7]{BrauchartHardinSaff2012} for the motivation behind this conjecture.

\begin{conjecture}[{\cite[Conjecture 5]{BrauchartHardinSaff2012}}]\label{conj:2-energy}
	The following asymptotic expansion holds:
	\begin{equation*}
		\minEnergy_2(\S^2,N)=\frac{1}{4}N^2\log{N}+CN^2+O(1),
	\end{equation*}
	with
	\begin{equation*}
		C=\frac{1}{4}(\gamma-\log(2\sqrt{3}\pi))+\frac{\sqrt{3}}{4\pi}(\gamma_1(2/3)-\gamma_1(1/3))=-0.0857\dotso<0,
	\end{equation*}
	where $\gamma$ is the Euler--Mascheroni constant and $\gamma_n(a)$ is the generalized Stieltjes constant appearing as the coefficient $\gamma_n(a)/n!$ of $(1-s)^n$ in the Laurent series expansion of the Hurwitz zeta function $\zeta(s,a)$ about $s=1$.
\end{conjecture}

Partial progress toward this conjecture was obtained in the same work (see \cite[Proposition 3]{BrauchartHardinSaff2012}), where the authors proved that
\begin{equation}\label{eq:original-bounds}
	-\frac{1}{4}N^2+O(N)\leq \minEnergy_2(\S^2,N)-\frac{1}{4}N^2\log N\leq \frac{1}{4}N^2\log \log N+O(N^2).
\end{equation}
The lower bound was recently improved by Brauchart in \cite{Brauchart2026}, where it is shown that
\begin{equation}\label{eq:lower-bound}
	\minEnergy_2(\S^2,N)\geq \frac{1}{4}N^2\log N+\frac{\gamma-1}{4}N^2+O(1),
\end{equation}
with
\begin{equation*}
	\frac{\gamma-1}{4}=-0.1056\ldots.
\end{equation*}
The upper bound in \eqref{eq:original-bounds} was first improved by Alishahi and Zamani in \cite[Corollary~1.4]{AlishahiZamani2015}, where the authors used the spherical ensemble to prove that, for any $N\geq 2$,
\begin{equation*}
	\minEnergy_2(\S^2,N)\leq \frac{1}{4}N^2\log N+\frac{\gamma}{4}N^2,
\end{equation*}
where
\begin{equation*}
	\frac{\gamma}{4}=0.1443\ldots.
\end{equation*}
More recently, de la Torre and Marzo \cite{delaTorreMarzo2024} extended the approach of Armentano, Beltrán, and Shub \cite{ArmentanoBeltranShub2011}, originally focused on the logarithmic energy, to other Riesz energies. In particular, they showed (see \cite[Corollary~4.1]{delaTorreMarzo2024}) that, for any $N\geq 2$,
\begin{equation*}
	\minEnergy_2(\S^2,N)\leq \frac{1}{4}N^2\log N+\frac{1}{4}\biggl(\frac{3}{2}-\log(2\pi)+\gamma\biggr)N^2,
\end{equation*}
where
\begin{equation*}
	\frac{1}{4}\biggl(\frac{3}{2}-\log(2\pi)+\gamma\biggr)=0.0598\ldots.
\end{equation*}
In this work, we further improve these upper bounds, obtaining for the first time a negative coefficient in the $N^2$ term and bringing the upper bound significantly closer to the conjectured value.

\subsection{Main results}

To improve the upper bound for the minimal $2$-energy, we study the expected $2$-energy of a specific realization of the Diamond ensemble originally introduced by Beltrán and Etayo \cite{BeltranEtayo2020}; see also \cite{BeltranEtayoLopezGomez2023} for a generalization of this construction and its extension to the real projective plane. 

In \cite{BeltranEtayo2020}, the authors define a family of points, the \emph{Diamond ensemble}, which, briefly speaking, is a sequence of sets of points on $\S^2$, depending on several parameters, that are placed on the sphere as follows. First, choose $p$ parallels. Then, for each parallel, indexed by $j$, choose $r_j$ points and place them at the vertices of a regular $r_j$-gon inscribed in that parallel rotated by a random phase $\theta_j\in[0,2\pi)$. Finally, add the north and south poles. 

In \cite[Section 3.2]{Etayo2021a}, Etayo proved that any configuration arising from the Diamond ensemble naturally induces an equal-area partition of the sphere. This partition consists of two spherical caps centered at the poles together with a collection of rectangular regions arranged in latitudinal collars, constructed so that each region contains exactly one point of the configuration. Consequently, the Diamond ensemble may be viewed as a variation of the zonal equal-area nodes introduced in \cite{RakhmanovSaffZhou1994}. The two constructions are closely related, though not identical, since the choice of the point within each cell is slightly different.

In \cite{BeltranEtayo2020,BeltranEtayoLopezGomez2023} the authors focused on the logarithmic energy of the Diamond ensemble and proved that this construction yields an expected logarithmic energy that is very close to the conjectured one. In this work we focus instead on the $2$-energy of this construction, obtaining the following result.

\begin{theorem}\label{thm:2-energy-diamond}
	For any integer $M\geq 1$, there exists an explicit set of $N=4M^2+2$ points on $\S^2$, depending on certain random parameters, whose expected $2$-energy is
	\begin{equation*}
		\frac{1}{4}N^2\log N+C_{\diamond,2}N^2+o(N^2),
	\end{equation*}
	where
	\begin{equation}\label{eq:C-diamond-2}
		C_{\diamond,2}=\frac{1}{2}\biggl(\gamma-\frac{1}{2}-\frac{\log 2}{3}\biggr)=-0.0769\ldots.
	\end{equation}
\end{theorem}

The specific construction leading to \cref{thm:2-energy-diamond} is described in \cref{sec:energy-4j}. Observe that \cref{thm:2-energy-diamond} provides an asymptotic upper bound for the minimal $2$-energy only when the number of points is of the form $N=4M^2+2$, with $M\geq 1$. Using the same idea as in \cite[Corollary 2]{BeltranMarzoOrtegaCerda2016}, we are able to extend the upper bound in \cref{thm:2-energy-diamond} to any value of $N$.

\begin{corollary}\label{cor:2-energy-diamond-all-N}
	For every $N\geq 2$, the following upper bound holds:
	\begin{equation*}
		\minEnergy_2(\S^2,N)\leq \frac{1}{4}N^2\log N+C_{\diamond,2}N^2+o(N^2),
	\end{equation*}
	where $C_{\diamond,2}$ is as in \eqref{eq:C-diamond-2}.
\end{corollary}

As we can see, the constant in the quadratic term of the upper bound is negative and differs from the conjectured constant in \cref{conj:2-energy} by only about 
$0.0088$. To the best of our knowledge, this is the first upper bound for the quadratic term in the asymptotic expansion of $\minEnergy_2(\S^2,N)$ with a negative coefficient. Consequently, combining the previous result and the lower bound in \eqref{eq:lower-bound}, we can state the following corollary.

\begin{corollary}
	There exist constants $C_1,C_2<0$ such that, for every $N\geq 2$,
	\begin{equation*}
		C_1N^2\leq \minEnergy_2(\S^2,N)-\frac{1}{4}N^2\log N\leq C_2N^2.
	\end{equation*}
\end{corollary}

The paper is organized as follows. In \cref{sec:general-construction}, we revisit the general construction underlying the Diamond ensemble and derive a formula for its expected $2$-energy. Next, in \cref{sec:energy-4j} we use a particular instance of this construction to prove the upper bound in \cref{thm:2-energy-diamond}. \Cref{sec:technical-proofs} is devoted to the proofs of several technical results. Finally, \cref{appendix:harmonic-numbers} collects some useful results on harmonic numbers used throughout \cref{sec:technical-proofs}.

\section{The Diamond ensemble and a formula for its expected $2$-energy}\label{sec:general-construction}

In this section, we briefly review the original construction introduced by Beltrán and Etayo in \cite{BeltranEtayo2020} and derive a closed formula for the $2$-energy of these configurations. Consider the following definition.

\begin{definition}\label{def:Omega}
	Let $\Omega(p,\set{r_j},\set{z_j})$ be the following set of points:
	\begin{equation*}
		\Omega(p,\set{r_j},\set{z_j})=\set{\mathcal{N}}\union\set{\mathcal{S}}\union\set{x_{j}^{\ell}}_{j=1,\dotsc,p}^{\ell=1,\dotsc,r_j},
	\end{equation*}
	where $\mathcal{N}=(0,0,1)$, $\mathcal{S}=(0,0,-1)$, and 
	\begin{equation*}
		x_{j}^{\ell}=\Biggl(\sqrt{1-z_j^2}\cos\biggl(\frac{2\pi \ell}{r_j}+\theta_j\biggr),\sqrt{1-z_j^2}\sin\biggl(\frac{2\pi \ell}{r_j}+\theta_j\biggr),z_j\Biggr).
	\end{equation*}
	Here, $p$ is the number of parallels, $r_j$ is the number of equally spaced points placed on the $j$th parallel, and $z_j$ is the height of the $j$th parallel, with $-1<z_j<1$, $1\leq j\leq p$, and $1\leq \ell \leq r_j$. Finally, $\theta_j$, with $0\leq \theta_j<2\pi$, denotes the random phase associated with the $j$th parallel. The total number of points in $\Omega(p,\set{r_j},\set{z_j})$ is
	\begin{equation*}
		N=2+\sum_{j=1}^{p} r_j.
	\end{equation*}
\end{definition} 

From \cref{def:Omega}, constructing a configuration from this general scheme amounts to specifying the number of parallels, their heights, and the number of points placed on each parallel. The following result provides a closed formula for the expected value of the $2$-energy of the points in $\Omega(p,\set{r_j},\set{z_j})$ with respect to the random parameters $\theta_1,\dotsc,\theta_p$.

\begin{proposition}\label{prop:energy-p}
	The expected value of the $2$-energy of configurations drawn from $\Omega(p,\set{r_j},\set{z_j})$ is
	\begin{align}\label{eq:energy-p}
		\MoveEqLeft\expectation[\Big]{\theta_1,\dotsc,\theta_p\in[0,2\pi]}{\Energy_{2}(\Omega(p,\set{r_j},\set{z_j}))}\\
		&=\frac{1}{2}+\frac{23}{12}\sum_{j=1}^{p}\frac{r_j}{1-z_j^2}+\frac{1}{12}\sum_{j=1}^{p} \frac{r_j^3}{1-z_j^2}
		+\frac{1}{2}\sum_{\substack{j,k=1\\ j\neq k}}^{p}\frac{r_jr_k}{\abs{z_j-z_k}}.\notag
	\end{align}
\end{proposition}

\begin{proof}
	See \cref{subsec:proof-prop:energy-p}.
\end{proof}

If we impose the additional symmetry condition that the configuration of parallels is symmetric with respect to the equator, while allowing each parallel to have its own independent random phase, we obtain the following result.

\begin{corollary}\label{cor:energy-M}
	Let $p=2M-1$ be an odd integer. If $r_j=r_{2M-j}$ and $z_j=-z_{2M-j}$ for $1\leq j\leq 2M-1$, then
	\begin{align*}
		\MoveEqLeft\expectation[\Big]{\theta_1,\dotsc,\theta_{2M-1}\in[0,2\pi]}{\Energy_{2}(\Omega(2M-1,\set{r_j},\set{z_j}))}\\
		&=\frac{1}{2}+\frac{23}{12}r_M+\frac{1}{12}r_M^3+\frac{23}{6}\sum_{j=1}^{M-1}\frac{r_j}{1-z_j^2}+\frac{1}{6}\sum_{j=1}^{M-1} \frac{r_j^3}{1-z_j^2}\\
		&\quad+2\sum_{j=1}^{M-1}\sum_{k=j+1}^{M}\frac{r_jr_k}{z_j-z_k}
		+\sum_{j=1}^{M-1}\sum_{k=1}^{M-1}\frac{r_jr_k}{z_j+z_k}.
	\end{align*}
\end{corollary}

\begin{proof}
	See \cref{subsec:proof-cor:energy-M}.
\end{proof}

As shown in \cite{BeltranEtayo2020}, in the case of the logarithmic energy it is possible to derive an explicit expression for the heights $z_j$ that minimize the corresponding expression for the expected energy. In the case of the $2$-energy, however, that is not possible. Therefore, we will use instead the optimal heights obtained in the logarithmic case (see \cite[Proposition 2.5]{BeltranEtayo2020}), defined as follows:
\begin{equation}\label{eq:heights-original}
	z_j=1-\frac{1+r_j+2\sum_{k=1}^{j-1}r_k}{N-1}.
\end{equation}
Hence, it only remains to choose the parameters $r_j$, that is, the number of points that we want to place on each parallel. In \cite{BeltranEtayo2020} (see also \cite{BeltranEtayoLopezGomez2023}), the authors explored different choices for these parameters, all of them based on defining the number of points $r_j$ using a piecewise linear function $r(x)$, with $0\leq x\leq 2M$, such that $r(j)$ gives the number of points on the $j$th parallel. In this work, we use the simplest choice considered there, namely $r_j=4j$ for $1\leq j\leq M$. 

Although more refined choices may lead to slightly improved results, they would require much more involved computations and the proofs would be less transparent. Moreover, numerical experiments indicate that the improvement is marginal. We refer the reader to \cite[Section 3.1]{BeltranEtayo2020} for the heuristic arguments motivating these choices of the parameters $r_j$.

\section[A specific realization of the Diamond ensemble. Proof of \cref{thm:2-energy-diamond}]{A specific realization of the Diamond ensemble.\\ Proof of \cref{thm:2-energy-diamond}}\label{sec:energy-4j}

In this section, we study the expected $2$-energy of a particular realization of the Diamond ensemble introduced in the previous section. More specifically, let $M$ be a positive integer, let $p=2M-1$, and let $r_j=4j$. Then, the resulting configuration consists of 
\begin{equation*}
	N=2+\sum_{j=1}^{2M-1} r_j=4M^2+2
\end{equation*}
points. As mentioned above, since no closed expression for the heights that minimize \eqref{eq:energy-p} is available for the $2$-energy, we will choose $z_j$ as the optimal heights for the logarithmic case, that is,
\begin{equation}\label{eq:heights}
	z_j=1-\frac{1+r_j+2\sum_{k=1}^{j-1}r_k}{N-1}=1-\frac{1+4j^2}{1+4M^2}=\frac{4M^2-4j^2}{1+4M^2}.
\end{equation}
\Cref{fig:diamond} illustrates a realization of this construction for $N=402$ points.

\begin{figure}[htbp]
	\centering
	\includegraphics[width=0.55\textwidth]{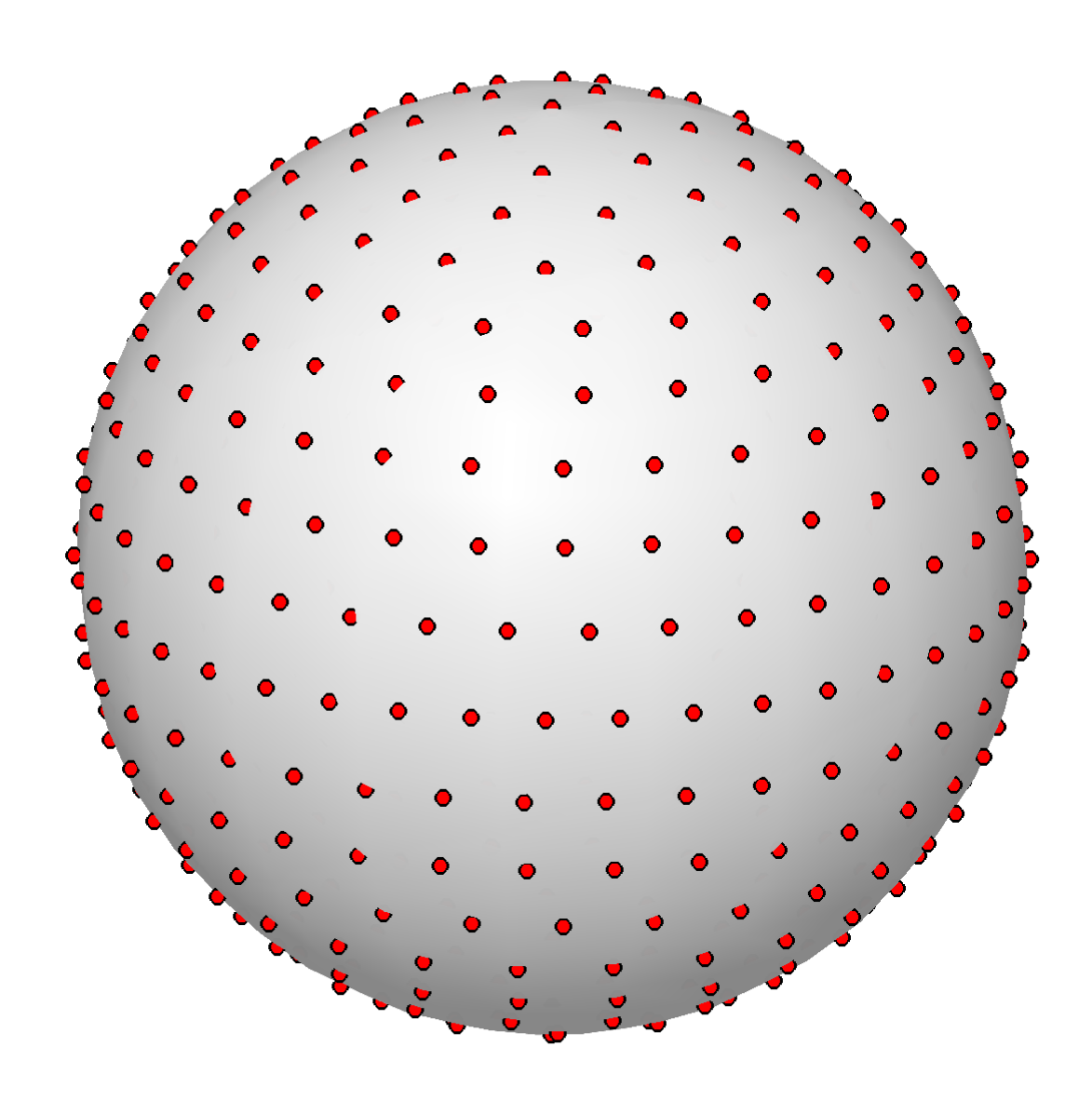}
	\caption{402 points on the sphere drawn from the Diamond ensemble with $r_j=4j$.}
	\label{fig:diamond}
\end{figure}

Choosing the heights $z_j$ as in \eqref{eq:heights}, one readily verifies that
\begin{equation}\label{eq:heights-1}
	\frac{1}{1-z_j^2}=\frac{1+4M^2}{2}\biggl(\frac{1}{1+4j^2}+\frac{1}{1+8M^2-4j^2}\biggr).
\end{equation}
Also, we have
\begin{align}
	z_j-z_k&=\frac{4M^2-4j^2}{1+4M^2}-\frac{4M^2-4k^2}{1+4M^2}=\frac{4(k^2-j^2)}{1+4M^2},\label{eq:heights-2}\\
	z_j+z_k&=\frac{4M^2-4j^2}{1+4M^2}+\frac{4M^2-4k^2}{1+4M^2}=\frac{4(2M^2-j^2-k^2)}{1+4M^2}.\label{eq:heights-3}
\end{align}
Hence, combining \cref{cor:energy-M} and \cref{eq:heights-1,eq:heights-2,eq:heights-3} we obtain the following result for the expected $2$-energy of the Diamond ensemble in this case.

\begin{corollary}\label{cor:expected-energy-4j}
	The expected $2$-energy of configurations drawn from the Diamond ensemble with $r_j=4j$ and $z_j$ given by \eqref{eq:heights} is
	\begin{align*}
		\MoveEqLeft\expectation[\Big]{\theta_1,\dotsc,\theta_{2M-1}\in[0,2\pi]}{\Energy_{2}(\Omega(2M-1,\set{r_j},\set{z_j}))}\\
		&=\frac{1}{2}+\frac{23}{3}M+\frac{16}{3}M^3\\
		&\quad+\frac{23}{3}(1+4M^2)\biggl(\sum_{j=1}^{M-1}\frac{j}{1+4j^2}+\sum_{j=1}^{M-1}\frac{j}{1+8M^2-4j^2}\biggr)\\
		&\quad+\frac{16}{3}(1+4M^2)\biggl(\sum_{j=1}^{M-1} \frac{j^3}{1+4j^2}+\sum_{j=1}^{M-1} \frac{j^3}{1+8M^2-4j^2}\biggr)\\
		&\quad+4(1+4M^2)\biggl(2\sum_{j=1}^{M-1}\sum_{k=j+1}^{M}\frac{jk}{k^2-j^2}+\sum_{j=1}^{M-1}\sum_{k=1}^{M-1}\frac{jk}{2M^2-j^2-k^2}\biggr).
	\end{align*}
\end{corollary}

As a consequence, to prove \cref{thm:2-energy-diamond} it remains to derive asymptotic estimates as $M\tendsto\infty$ for the sums in \cref{cor:expected-energy-4j}. For the reader’s convenience, we have encapsulated the necessary technical results as lemmas and deferred the proofs to \cref{subsec:proofs-sec:energy-4j}. Denote
\begin{align*}
	S_1&=\sum_{j=1}^{M-1}\frac{j}{1+4j^2},\\
	S_2&=\sum_{j=1}^{M-1}\frac{j}{1+8M^2-4j^2},\\
	S_3&=\sum_{j=1}^{M-1} \frac{j^3}{1+4j^2},\\
	S_4&=\sum_{j=1}^{M-1} \frac{j^3}{1+8M^2-4j^2},\\
	S_5&=\sum_{j=1}^{M-1}\sum_{k=j+1}^{M}\frac{jk}{k^2-j^2},\\
	S_6&=\sum_{j=1}^{M-1}\sum_{k=1}^{M-1}\frac{jk}{2M^2-j^2-k^2}.\\
\end{align*}

\begin{lemma}\label{lemma:S1}
	The following estimate holds:
	\begin{equation*}
		S_1=O(\log M),\qquad M\tendsto\infty.
	\end{equation*}
\end{lemma}

\begin{proof}
	See \cref{subsubsec:proof-lemma:S1}.
\end{proof}

\begin{lemma}\label{lemma:S2}
	The following estimate holds:
	\begin{equation*}
		S_2=O(1),\qquad M\tendsto\infty.
	\end{equation*}
\end{lemma}

\begin{proof}
	See \cref{subsubsec:proof-lemma:S2}.
\end{proof}

\begin{lemma}\label{lemma:S3}
	The following estimate holds:
	\begin{equation*}
		S_3=\frac{M^2}{8}+O(M),\qquad M\tendsto\infty.
	\end{equation*}
\end{lemma}

\begin{proof}
	See \cref{subsubsec:proof-lemma:S3}.
\end{proof}

\begin{lemma}\label{lemma:S4}
	The following estimate holds:
	\begin{equation*}
		S_4=\biggl(\frac{\log 2}{4}-\frac{1}{8}\biggr)M^2+o(M^2),\qquad M\tendsto\infty.
	\end{equation*}
\end{lemma}

\begin{proof}
	See \cref{subsubsec:proof-lemma:S4}.
\end{proof}

\begin{lemma}\label{lemma:S5}
	The following estimate holds:
	\begin{equation*}
		S_5=\frac{1}{4}M^2\log M+\frac{M^2}{4}\biggl(\gamma-\frac{1}{2}-\log 2\biggr)+O(M\log M),\qquad M\tendsto\infty,
	\end{equation*}
	where $\gamma$ is the Euler--Mascheroni constant.
\end{lemma}

\begin{proof}
	See \cref{subsubsec:proof-lemma:S5}.
\end{proof}

\begin{lemma}\label{lemma:S6}
	The following estimate holds:
	\begin{equation*}
		S_6=\frac{M^2}{2}\log{2}+o(M^2),\qquad M\tendsto\infty.
	\end{equation*}
\end{lemma}

\begin{proof}
	See \cref{subsubsec:proof-lemma:S6}.
\end{proof}

\begin{proof}[Proof of \cref{thm:2-energy-diamond}]
	From \cref{cor:expected-energy-4j} and \cref{lemma:S1,lemma:S2,lemma:S3,lemma:S4,lemma:S5,lemma:S6}, we have
	\begin{align*}
		\MoveEqLeft\expectation[\Big]{\theta_1,\dotsc,\theta_{2M-1}\in[0,2\pi]}{\Energy_{2}(\Omega(2M-1,\set{r_j},\set{z_j}))}\\
		&=\frac{64}{3}\biggl(\frac{1}{8}+\frac{\log 2}{4}-\frac{1}{8}\biggr)M^4\\
		&\quad +32M^2\biggl(\frac{1}{4}M^2\log M+\frac{M^2}{4}\biggl(\gamma-\frac{1}{2}-\log 2\biggr)\biggr)\\
		&\quad+8\log(2)M^4+o(M^4)\\
		&=8M^4\log M+\biggl(\frac{16}{3}\log 2+8\gamma-4\biggr)M^4+o(M^4).
	\end{align*}
	Finally, using that $N=4M^2+2$, the theorem follows.
\end{proof}

\section{Technical proofs}\label{sec:technical-proofs}

\subsection{Proof of \cref{prop:energy-p}}\label{subsec:proof-prop:energy-p}

As in \cite[Section 5.2]{BeltranEtayo2020}, to derive the expression for the expected $2$-energy associated with the set $\Omega(p,\set{r_j},\set{z_j})$ we need to compute the following three quantities:
\begin{enumerate}
	\item $A$: The energy between every point and the poles, plus the energy between the poles;
	
	\item $B$: the energy of the scaled roots of unity for every parallel; and
	
	\item $C$: the energy between the points of every pair of parallels.
\end{enumerate}

We compute these quantities in the following subsections.

\subsection*{Computation of the quantity $A$}

The energy between the north pole and the south pole is
\begin{equation*}
	\norm{(0,0,1)-(0,0,-1)}^{-2}=\frac{1}{4},
\end{equation*}
and we have to add this contribution twice. To compute the energy between a generic point $x_{j}^{\ell}$ of our set and the poles, note that
\begin{align*}
	\norm{(0,0,1)-x_j^{\ell}}&=\sqrt{2}\sqrt{1-z_j},\\
	\norm{(0,0,-1)-x_j^{\ell}}&=\sqrt{2}\sqrt{1+z_j}.
\end{align*}
Then,
\begin{equation}\label{eq:A}
	A=\frac{1}{2}+\sum_{j=1}^{p}r_j\frac{1}{1-z_j}+\sum_{j=1}^{p}r_j\frac{1}{1+z_j}=\frac{1}{2}+2\sum_{j=1}^{p}\frac{r_j}{1-z_j^2}.
\end{equation}

\subsection*{Computation of the quantity $B$}

To compute the $2$-energy of the $r_j$ scaled roots of unity on the parallel of height $z_j$, we use the following result.

\begin{lemma}[{see \cite[Theorem 1.1 and Remark 1]{BrauchartHardinSaff2009}}]\label{lemma:2-energy-roots-unity}
	Let $\omega_r$ be a configuration of $r$ equally spaced points on $\S^1$. Then,
	\begin{equation*}
		\Energy_{2}(\omega_r)=\frac{1}{12}r(r^2-1).
	\end{equation*}
\end{lemma}

From \cref{lemma:2-energy-roots-unity}, the $2$-energy of equally spaced points on a circle of radius $R$ can be easily computed:
\begin{equation*}
	\Energy_{2}^{R}(\omega_r)=\frac{1}{R^2}\Energy_{2}(\omega_r).
\end{equation*}
Hence, since the parallel of height $z_j$ is a circle of radius $\sqrt{1-z_j^2}$, we have
\begin{equation}\label{eq:B}
	B=\frac{1}{12}\sum_{j=1}^{p} \frac{r_j(r_j^2-1)}{1-z_j^2}.
\end{equation}

\subsection*{Computation of the quantity $C$}

To obtain an expression for the energy between points on different parallels of the sphere we will need the following lemma.

\begin{lemma}\label{lemma:integral}
	The following equality holds:
	\begin{equation*}
		\int_{0}^{\pi}\frac{1}{1+a\cos\theta}\dd\theta=\frac{\pi}{\sqrt{1-a^2}},\qquad a^2< 1.
	\end{equation*}
\end{lemma}

\begin{proof}
	See \cite[3.613-1]{GradshteynRyzhik2007}.
\end{proof}

\begin{proposition}\label{prop:crossed-energy}
	Let $x$ and $y$ be two points sampled independently and uniformly on the parallels of height $z_j$ and $z_k$, respectively, with $z_j\neq z_k$. Then, the expected value of $\norm{x-y}^{-2}$ is
	\begin{equation*}
		\frac{1}{2\abs{z_j-z_k}}.
	\end{equation*}
\end{proposition}

\begin{proof}
	By rotational invariance, we may fix
	\begin{equation*}
		x=\Bigl(\sqrt{1-z_j^2},0,z_j\Bigr)
	\end{equation*}
	and write
	\begin{equation*}
		y=\Bigl(\sqrt{1-z_k^2}\cos\theta,\sin\theta,z_k\Bigr),
	\end{equation*}
	where $\theta$ is uniformly distributed in $[0,2\pi]$. Then,
	\begin{equation*}
		\norm{x-y}^2=2-2\innerprod{x}{y}=2\Bigl(1-z_jz_k-\sqrt{1-z_j^2}\sqrt{1-z_k^2}\cos\theta\Bigr).
	\end{equation*}
	Hence,
	\begin{equation*}
		\expectation[\Big]{\theta\in[0,2\pi]}{\norm{x-y}^{-2}}
		=
		\frac{1}{2\pi}\int_0^{2\pi}
		\frac{1}{2\Bigl(1-z_jz_k-\sqrt{1-z_j^2}\sqrt{1-z_k^2}\cos\theta\Bigr)}\dd\theta.
	\end{equation*}
	Using the symmetry of the integrand, this becomes
	\begin{equation*}
		\expectation[\Big]{\theta\in[0,2\pi]}{\norm{x-y}^{-2}}
		=\frac{1}{2\pi}\int_{0}^{\pi}\frac{1}{1-z_jz_k-\sqrt{1-z_j^2}\sqrt{1-z_k^2}\cos\theta}\dd\theta.
	\end{equation*}
	Set
	\begin{equation*}
		\alpha=1-z_jz_k,\qquad \beta=-\sqrt{1-z_j^2}\sqrt{1-z_k^2}.
	\end{equation*}
	Then,
	\begin{equation*}
		\expectation[\Big]{\theta\in[0,2\pi]}{\norm{x-y}^{-2}}
		=\frac{1}{2\pi}\int_{0}^{\pi}\frac{1}{\alpha+\beta\cos\theta}\dd\theta
		=\frac{1}{2\pi\alpha}\int_{0}^{\pi}\frac{1}{1+(\beta/\alpha)\cos\theta}\dd\theta.
	\end{equation*}
	One readily verifies that $(\beta/\alpha)^2<1$. Then, from \cref{lemma:integral} we have
	\begin{equation*}
		\expectation[\Big]{\theta\in[0,2\pi]}{\norm{x-y}^{-2}}=\frac{1}{2\alpha\sqrt{1-\beta^2/\alpha^2}}
		=\frac{1}{2\sqrt{\alpha^2-\beta^2}}
		=\frac{1}{2\abs{z_j-z_k}}.\qedhere
	\end{equation*}
\end{proof}

The following result is a direct consequence of \cref{prop:crossed-energy}.

\begin{corollary}\label{cor:energy-between-parallels}
	Let $x_j^{\ell}$ be as in \cref{def:Omega}. Then, for $j\neq k$,
	\begin{equation*}
		\expectation[\Big]{\theta_j,\theta_k\in[0,2\pi]}{\sum_{\ell=1}^{r_j}\sum_{m=1}^{r_k}\norm{x_j^{\ell}-x_k^{m}}^{-2}}=\frac{r_jr_k}{2\abs{z_j-z_k}},
	\end{equation*}
	where $\theta_j$ and $\theta_k$ are uniformly distributed in $[0,2\pi]$.
\end{corollary}

From \cref{cor:energy-between-parallels}, we have
\begin{equation}\label{eq:C}
	C=\sum_{\substack{j,k=1\\ j\neq k}}^{p}\frac{r_jr_k}{2\abs{z_j-z_k}}.
\end{equation}

\begin{proof}[Proof of \cref{prop:energy-p}]
	Since
	\begin{equation*}
		\expectation[\Big]{\theta_1,\dotsc,\theta_p\in[0,2\pi]}{\Energy_{2}(\Omega(p,\set{r_j},\set{z_j}))}=A+B+C,
	\end{equation*}
	the proposition follows by combining \cref{eq:A,eq:B,eq:C}.
\end{proof}

\subsection{Proof of \cref{cor:energy-M}}\label{subsec:proof-cor:energy-M}

From \cref{prop:energy-p}, we have
\begin{align*}
	\MoveEqLeft\expectation[\Big]{\theta_1,\dotsc,\theta_{2M-1}\in[0,2\pi]}{\Energy_{2}(\Omega(2M-1,\set{r_j},\set{z_j}))}\\
	&=\frac{1}{2}+\frac{23}{12}\sum_{j=1}^{2M-1}\frac{r_j}{1-z_j^2}+\frac{1}{12}\sum_{j=1}^{2M-1} \frac{r_j^3}{1-z_j^2}
	+\frac{1}{2}\sum_{\substack{j,k=1\\ j\neq k}}^{2M-1}\frac{r_jr_k}{\abs{z_j-z_k}}\\
	&=\frac{1}{2}+\frac{23}{12}r_M+\frac{1}{12}r_M^3+\frac{23}{6}\sum_{j=1}^{M-1}\frac{r_j}{1-z_j^2}+\frac{1}{6}\sum_{j=1}^{M-1} \frac{r_j^3}{1-z_j^2}
	+\frac{1}{2}\sum_{\substack{j,k=1\\ j\neq k}}^{2M-1}\frac{r_jr_k}{\abs{z_j-z_k}}.
\end{align*}
To complete the proof of \cref{cor:energy-M}, it remains to simplify the double sum in the previous expression. We have
\begin{align*}
	\sum_{\substack{j,k=1\\ j\neq k}}^{2M-1}\frac{r_jr_k}{\abs{z_j-z_k}}&=\sum_{j=1}^{2M-2}\sum_{k=j+1}^{2M-1}\frac{r_jr_k}{z_j-z_k}+\sum_{j=2}^{2M-1}\sum_{k=1}^{j-1}\frac{r_jr_k}{z_k-z_j}\\
	&=2\sum_{j=1}^{2M-2}\sum_{k=j+1}^{2M-1}\frac{r_jr_k}{z_j-z_k}.
\end{align*}
To simplify this expression, we decompose the sum as follows:
\begin{align*}
	\sum_{j=1}^{2M-2}\sum_{k=j+1}^{2M-1}\frac{r_jr_k}{z_j-z_k}
	&=\sum_{j=1}^{M-1}\sum_{k=j+1}^{2M-1}\frac{r_jr_k}{z_j-z_k}
	+\sum_{j=M}^{2M-2}\sum_{k=j+1}^{2M-1}\frac{r_jr_k}{z_j-z_k}\\
	&=\sum_{j=1}^{M-1}\sum_{k=j+1}^{M}\frac{r_jr_k}{z_j-z_k}
	+\sum_{j=1}^{M-1}\sum_{k=M+1}^{2M-1}\frac{r_jr_k}{z_j-z_k}\\
	&\quad +\sum_{j=M}^{2M-2}\sum_{k=j+1}^{2M-1}\frac{r_jr_k}{z_j-z_k}.
\end{align*}
We have
\begin{equation*}
	\sum_{j=1}^{M-1}\sum_{k=M+1}^{2M-1}\frac{r_jr_k}{z_j-z_k}=	\sum_{j=1}^{M-1}\sum_{k=1}^{M-1}\frac{r_jr_k}{z_j+z_k},
\end{equation*}
and
\begin{equation*}
	\sum_{j=M}^{2M-2}\sum_{k=j+1}^{2M-1}\frac{r_jr_k}{z_j-z_k}=\sum_{j=2}^{M}\sum_{k=2M-j+1}^{2M-1}\frac{r_{j}r_k}{-z_{j}-z_k}=\sum_{j=2}^{M}\sum_{k=1}^{j-1}\frac{r_{j}r_k}{z_k-z_{j}}.
\end{equation*}
Up to this point, we have proved that
\begin{equation*}
	\frac{1}{2}\sum_{\substack{j,k=1\\ j\neq k}}^{2M-1}\frac{r_jr_k}{\abs{z_j-z_k}}
	=\sum_{j=1}^{M-1}\sum_{k=j+1}^{M}\frac{r_jr_k}{z_j-z_k}
	+\sum_{j=1}^{M-1}\sum_{k=1}^{M-1}\frac{r_jr_k}{z_j+z_k}
	+\sum_{j=2}^{M}\sum_{k=1}^{j-1}\frac{r_{j}r_k}{z_k-z_{j}}.
\end{equation*}
Since
\begin{equation*}
	\sum_{j=2}^{M}\sum_{k=1}^{j-1}\frac{r_{j}r_k}{z_k-z_{j}}=\sum_{j=1}^{M-1}\sum_{k=j+1}^{M}\frac{r_{j}r_k}{z_j-z_{k}},
\end{equation*}
it follows that
\begin{equation*}
	\frac{1}{2}\sum_{\substack{j,k=1\\ j\neq k}}^{2M-1}\frac{r_jr_k}{\abs{z_j-z_k}}
	=2\sum_{j=1}^{M-1}\sum_{k=j+1}^{M}\frac{r_jr_k}{z_j-z_k}
	+\sum_{j=1}^{M-1}\sum_{k=1}^{M-1}\frac{r_jr_k}{z_j+z_k}.
\end{equation*}
This proves the corollary.\qed

\subsection{Proofs of the auxiliary lemmas in \cref{sec:energy-4j}}\label{subsec:proofs-sec:energy-4j}

\subsubsection{Proof of \cref{lemma:S1}}\label{subsubsec:proof-lemma:S1}

We have
\begin{equation*}
	S_1=\sum_{j=1}^{M-1}\frac{j}{1+4j^2}\leq\frac{1}{4}\sum_{j=1}^{M-1}\frac{1}{j}=\frac{1}{4}H_{M-1}=O(\log M),
\end{equation*}
where $H_{M-1}$ denotes the $(M-1)$th harmonic number.\qed

\subsubsection{Proof of \cref{lemma:S2}}\label{subsubsec:proof-lemma:S2}

We have
\begin{equation*}
	S_2=\sum_{j=1}^{M-1}\frac{j}{1+8M^2-4j^2}\leq \frac{1}{4M^2}\sum_{j=1}^{M-1}j=O(1).\pushQED{\qed}\qedhere
\end{equation*}

\subsubsection{Proof of \cref{lemma:S3}}\label{subsubsec:proof-lemma:S3}

We have
\begin{equation*}
	S_3=\sum_{j=1}^{M-1} \frac{j^3}{1+4j^2}.
\end{equation*}
Note that
\begin{equation*}
	\frac{j^3}{1+4j^2}=\frac{1}{4}\biggl(j-\frac{j}{1+4j^2}\biggr).
\end{equation*}
Then,
\begin{equation*}
	S_3=\frac{1}{4}\sum_{j=1}^{M-1}j-\frac{1}{4}S_1=\frac{M^2}{8}+O(M),
\end{equation*}
where we have used \cref{lemma:S1}.\qed

\subsubsection{Proof of \cref{lemma:S4}}\label{subsubsec:proof-lemma:S4}

We have
\begin{equation*}
	S_4=\sum_{j=1}^{M-1} \frac{j^3}{1+8M^2-4j^2}=\sum_{j=0}^{M-1} \frac{j^3}{1+8M^2-4j^2}.
\end{equation*}
Note that
\begin{equation*}
	\frac{S_4}{M^2}=\frac{1}{M}\sum_{j=0}^{M-1}\phi_M\biggl(\frac{j}{M}\biggr),
\end{equation*}
where
\begin{equation*}
	\phi_M(x)=\frac{x^3}{8-4x^2+1/M^2}.
\end{equation*}
Define 
\begin{equation*}
	\phi(x)=\frac{x^3}{8-4x^2},\qquad x\in[0,1].
\end{equation*}
Observe that $\phi_M$ converges uniformly to $\phi$ as $M\tendsto\infty$. Therefore,
\begin{equation*}
	\abs[\bigg]{\frac{1}{M}\sum_{j=0}^{M-1}\biggl(\phi_M\biggl(\frac{j}{M}\biggr)-\phi\biggl(\frac{j}{M}\biggr)\biggr)}\leq \norm{\phi_M-\phi}_{\infty}\xrightarrow[M\tendsto\infty]{}0.
\end{equation*}
Hence,
\begin{equation*}
	\frac{S_4}{M^2}=\frac{1}{M}\sum_{j=0}^{M-1}\phi\biggl(\frac{j}{M}\biggr)+o(1)\quad \text{as $M\tendsto\infty$}.
\end{equation*}
Since $\phi$ is continuous on $[0,1]$ and 
\begin{equation*}
	\frac{1}{M}\sum_{j=0}^{M-1}\phi\biggl(\frac{j}{M}\biggr)
\end{equation*}
is a Riemann sum for $\phi$ on that interval, we obtain
\begin{equation*}
	\lim_{M\tendsto\infty}\frac{1}{M}\sum_{j=0}^{M-1}\phi\biggl(\frac{j}{M}\biggr)=\int_{0}^{1}\phi(x)\dd x=\int_{0}^{1}\frac{x^3}{8-4x^2}\dd x=\frac{\log 2}{4}-\frac{1}{8}.
\end{equation*}
Therefore,
\begin{equation*}
	\lim_{M\tendsto\infty}\frac{S_4}{M^2}=\frac{\log 2}{4}-\frac{1}{8}.\pushQED{\qed}\qedhere
\end{equation*}

\subsubsection{Proof of \cref{lemma:S5}}\label{subsubsec:proof-lemma:S5}

We have
\begin{equation*}
	S_5=\sum_{j=1}^{M-1}\sum_{k=j+1}^{M}\frac{jk}{k^2-j^2}.
\end{equation*}
Note that
\begin{equation*}
	\frac{jk}{k^2-j^2}=\frac{j}{2}\biggl(\frac{1}{k-j}+\frac{1}{k+j}\biggr).
\end{equation*}
Since
\begin{equation*}
	\sum_{k=j+1}^{M}\frac{1}{k-j}+\sum_{k=j+1}^{M}\frac{1}{k+j}=\sum_{k=1}^{M-j}\frac{1}{k}+\sum_{k=2j+1}^{M+j}\frac{1}{k}=H_{M-j}+H_{M+j}-H_{2j},
\end{equation*}
where $H_n$ denotes the $n$th harmonic number (see \cref{appendix:harmonic-numbers}), we have
\begin{equation}\label{eq:S5-sums}
	S_5=\frac{1}{2}\sum_{j=1}^{M-1}j\Bigl(H_{M-j}+H_{M+j}-H_{2j}\Bigr)
	=\frac{1}{2}\Bigl(S_5^{(1)}+S_5^{(2)}-S_5^{(3)}\Bigr),
\end{equation}
where
\begin{align*}
	S_5^{(1)}&=\sum_{j=1}^{M-1}jH_{M-j},\\
	S_5^{(2)}&=\sum_{j=1}^{M-1}jH_{M+j},\\
	S_5^{(3)}&=\sum_{j=1}^{M-1}jH_{2j}.
\end{align*}
From \cref{lemma:sumHk,lemma:sumkHk} and \cref{eq:asymptotic-harmonic-number}, we have
\begin{align*}
	S_5^{(1)}&=\sum_{j=1}^{M-1}jH_{M-j}\\
	&=\sum_{j=1}^{M-1}(M-j)H_{j}\\
	&=M\sum_{j=1}^{M-1}H_j-\sum_{j=1}^{M-1}jH_j\\
	&=\frac{1}{2}M^2\log M+\biggl(\frac{\gamma}{2}-\frac{3}{4}\biggr)M^2+O(M\log M).
\end{align*}
Similarly,
\begin{align*}
	S_5^{(2)}&=\sum_{j=1}^{M-1}jH_{M+j}\\
	&=\sum_{j=M+1}^{2M-1}(j-M)H_{j}\\
	&=\sum_{j=1}^{2M-1}(j-M)H_{j}-\sum_{j=1}^{M}(j-M)H_{j}\\
	&=\sum_{j=1}^{2M-1}jH_j-M\sum_{j=1}^{2M-1}H_j-\sum_{j=1}^{M}jH_j+M\sum_{j=1}^{M}H_j\\
	&=\frac{1}{2}M^2\log M+M^2\biggl(\frac{\gamma}{2}+\frac{1}{4}\biggr)+O(M\log M).
\end{align*}
Finally, from \cref{lemma:sum2kH2k} and \cref{eq:asymptotic-harmonic-number}, we have
\begin{align*}
	S_5^{(3)}&=\frac{1}{2}\sum_{j=1}^{M-1}2jH_{2j}\\
	&=\frac{1}{2}M^2\log M+\biggl(\frac{\gamma}{2}+\frac{\log 2}{2}-\frac{1}{4}\biggr)M^2+O(M\log M).
\end{align*}
From \eqref{eq:S5-sums}, combining the previous expressions the lemma follows. \qed

To prove \cref{lemma:S6} we will use the following auxiliary result.

\begin{lemma}\label{lemma:aux-lemma-integral}
	Let $f\from [0,1)^2\to \R$ be nonnegative, continuous, and increasing in both variables. Assume that $f$ has an integrable singularity in $(1,1)$. Then,
	\begin{equation*}
		\frac{1}{M^2}\sum_{j=1}^{M-1}\sum_{k=1}^{M-1}f\biggl(\frac{j}{M},\frac{k}{M}\biggr)\tendsto\int_{0}^{1}\int_{0}^{1}f(x,y)\dd x\dd y\quad \text{as $M\tendsto\infty$}.
	\end{equation*}
\end{lemma}

\begin{proof}
	Since $f$ is increasing in both variables, we have
	\begin{equation*}
		\int_{(j-1)/M}^{j/M}\int_{(k-1)/M}^{k/M}f(x,y)\dd x\dd y\leq \frac{1}{M^2}f\biggl(\frac{j}{M},\frac{k}{M}\biggr).
	\end{equation*}
	Then,
	\begin{equation*}
		\sum_{j=1}^{M-1}\sum_{k=1}^{M-1}\int_{(j-1)/M}^{j/M}\int_{(k-1)/M}^{k/M}f(x,y)\leq \sum_{j=1}^{M-1}\sum_{k=1}^{M-1}\frac{1}{M^2}f\biggl(\frac{j}{M},\frac{k}{M}\biggr),
	\end{equation*}
	that is,
	\begin{equation*}
		\int_{0}^{1-1/M}\int_{0}^{1-1/M}f(x,y)\dd x\dd y \leq \frac{1}{M^2}\sum_{j=1}^{M-1}\sum_{k=1}^{M-1}f\biggl(\frac{j}{M},\frac{k}{M}\biggr).
	\end{equation*}
	Similarly, since
	\begin{equation*}
		\frac{1}{M^2}f\biggl(\frac{j}{M},\frac{k}{M}\biggr)\leq \int_{j/M}^{(j+1)/M}\int_{k/M}^{(k+1)/M}f(x,y)\dd x\dd y,
	\end{equation*}
	we have
	\begin{equation*}
		\frac{1}{M^2}\sum_{j=1}^{M-1}\sum_{k=1}^{M-1}f\biggl(\frac{j}{M},\frac{k}{M}\biggr)\leq \int_{1/M}^{1}\int_{1/M}^{1}f(x,y)\dd x\dd y\leq \int_{0}^{1}\int_{0}^{1}f(x,y)\dd x\dd y.
	\end{equation*}
	Hence, we obtain
	\begin{equation*}
		\int_{0}^{1-1/M}\int_{0}^{1-1/M}f(x,y)\dd x\dd y \leq\frac{1}{M^2}\sum_{j=1}^{M-1}\sum_{k=1}^{M-1}f\biggl(\frac{j}{M},\frac{k}{M}\biggr)\leq \int_{0}^{1}\int_{0}^{1}f(x,y)\dd x\dd y.
	\end{equation*}
	Finally, since $f$ is integrable in $(1,1)$, the lemma follows by using the squeeze theorem.
\end{proof}
	
\subsubsection{Proof of \cref{lemma:S6}}\label{subsubsec:proof-lemma:S6}

We have
\begin{align*}
	S_6&=\sum_{j=1}^{M-1}\sum_{k=1}^{M-1}\frac{jk}{2M^2-j^2-k^2}\\
	&=\sum_{j=1}^{M-1}\sum_{k=1}^{M-1}\frac{(j/M)(k/M)}{2-(j/M)^2-(k/M)^2}\\
	&=\sum_{j=1}^{M-1}\sum_{k=1}^{M-1}f\biggl(\frac{j}{M},\frac{k}{M}\biggr),
\end{align*}	
where
\begin{equation*}
	f(x,y)=\frac{xy}{2-x^2-y^2}.
\end{equation*}
Using \cref{lemma:aux-lemma-integral}, we have
\begin{equation*}
	\lim_{M\to\infty}\frac{S_6}{M^2}
	=\int_0^1\int_0^1f(x,y)\dd x\dd y.
\end{equation*}
Finally, a direct computation yields
\begin{equation*}
	\int_0^1\int_0^1f(x,y)\dd x\dd y=\int_0^1\int_0^1\frac{xy}{2-x^2-y^2}\dd x\dd y
	=\frac{\log 2}{2},
\end{equation*}
which completes the proof. \qed

\appendix

\section{Harmonic numbers}\label{appendix:harmonic-numbers}

The \emph{$n$th harmonic number $H_n$} is defined as the sum of the reciprocals of the first $n$ natural numbers:
\begin{equation*}
	H_n\defined \sum_{k=1}^{n}\frac{1}{k}.
\end{equation*}
Harmonic numbers satisfy the following recurrence relation:
\begin{equation}\label{eq:harmonic-recurrence}
	H_{n+1}=H_n+\frac{1}{n+1}.
\end{equation}
The following asymptotic expansion for the harmonic numbers is also well known; see, for example, \cite[Eq.~(9.28)]{GrahamKnuthPatashnik1994}:
\begin{equation}\label{eq:asymptotic-harmonic-number}
	H_n=\log n+\gamma+O(1/n),
\end{equation}
where $\gamma$ is the Euler--Mascheroni constant.
 
\subsection*{Some finite sums involving harmonic numbers}

Of the following four results, the first two are well known; see \cite[Eqs. (6.67) and (6.68)]{GrahamKnuthPatashnik1994}. For the remaining two we provide a proof.

\begin{lemma}\label{lemma:sumHk}
	The following equality holds:
	\begin{equation*}
		\sum_{k=1}^{n}H_k=(n+1)H_n-n.
	\end{equation*}
\end{lemma}

\begin{lemma}\label{lemma:sumkHk}
	The following equality holds:
	\begin{equation*}
		\sum_{k=1}^{n}kH_k=\frac{n(n+1)}{2}H_n-\frac{n(n-1)}{4}.
	\end{equation*}
\end{lemma}

\begin{lemma}\label{lemma:sumH2k}
	The following equality holds:
	\begin{equation*}
		\sum_{k=1}^{n}H_{2k}=\frac{2n+1}{2}H_{2n}+\frac{1}{4}H_n-n.
	\end{equation*}
\end{lemma}

\begin{proof}
	Let $S$ be the sum in the lemma. Note that
	\begin{equation}\label{eq:aux-1}
		S=\sum_{k=1}^{n}H_{2k}=\sum_{k=1}^{2n}H_{k}-\sum_{k=1}^{n}H_{2k-1}=(2n+1)H_{2n}-2n-\sum_{k=1}^{n}H_{2k-1},
	\end{equation}
	where we have used \cref{lemma:sumHk}. Since
	\begin{equation*}
		H_{2k-1}=H_{2k}-\frac{1}{2k},
	\end{equation*}
	we have
	\begin{equation}\label{eq:aux-2}
		\sum_{k=1}^{n}H_{2k-1}=\sum_{k=1}^{n}\biggl(H_{2k}-\frac{1}{2k}\biggr)=\sum_{k=1}^{n}H_{2k}-\frac{1}{2}H_n=S-\frac{1}{2}H_n.
	\end{equation}
	Then, from \eqref{eq:aux-1} and \eqref{eq:aux-2} we obtain
	\begin{equation*}
		S=(2n+1)H_{2n}-2n-S+\frac{1}{2}H_n,
	\end{equation*}
	and therefore
	\begin{equation*}
		S=\frac{2n+1}{2}H_{2n}+\frac{1}{4}H_n-n.\qedhere
	\end{equation*}
\end{proof}

\begin{lemma}\label{lemma:sum2kH2k}
	The following equality holds:
	\begin{equation*}
		\sum_{k=1}^{n}2kH_{2k}=(n+1/2)^2H_{2n}-\frac{1}{8}H_n-\frac{n(2n-1)}{4}.
	\end{equation*}
\end{lemma}

\begin{proof}
	Let $S$ be the sum in the lemma. Note that
	\begin{align*}
		S&=\sum_{k=1}^{2n}kH_{k}-\sum_{k=1}^{n}(2k-1)H_{2k-1}\\
		&=n(2n+1)H_{2n}-\frac{n(2n-1)}{2}-\sum_{k=1}^{n}(2k-1)H_{2k-1},
	\end{align*}
	where we have used \cref{lemma:sumkHk}. Using \eqref{eq:harmonic-recurrence} and \cref{lemma:sumHk,lemma:sumH2k}, we have
	\begin{align*}
		\sum_{k=1}^{n}(2k-1)H_{2k-1}&=\sum_{k=1}^{n}(2k-1)\biggl(H_{2k}-\frac{1}{2k}\biggr)\\
		&=\sum_{k=1}^{n}2kH_{2k}-\sum_{k=1}^{n}H_{2k}-n+\sum_{k=1}^{n}\frac{1}{2k}\\
		&=S-\sum_{k=1}^{n}H_{2k}-n+\frac{1}{2}H_n\\
		&=S-\biggl(\frac{2n+1}{2}H_{2n}+\frac{1}{4}H_n-n\biggr)-n+\frac{1}{2}H_n\\
		&=S-\frac{2n+1}{2}H_{2n}+\frac{1}{4}H_n.
	\end{align*}
	Hence,
	\begin{equation*}
		S=n(2n+1)H_{2n}-\frac{n(2n-1)}{2}-\biggl(S-\frac{2n+1}{2}H_{2n}+\frac{1}{4}H_n\biggr),
	\end{equation*}
	and therefore
	\begin{equation*}
		S=(n+1/2)^2H_{2n}-\frac{1}{8}H_n-\frac{n(2n-1)}{4}.\qedhere
	\end{equation*}
\end{proof}

	
 \enlargethispage{\baselineskip}
\bibliographystyle{amsplain}

\end{document}